\documentclass[12pt]{article}
\nonstopmode
\RequirePackage[colorlinks,citecolor=blue,urlcolor=blue,linkcolor=blue]{hyperref}
\hypersetup{
colorlinks = true,
citecolor=blue,
urlcolor=blue,
linkcolor=blue,
pdfauthor = {Alexey Kuznetsov},
pdfkeywords = {Gaver-Stehfest algorithm, inverse Laplace transform, Lambert W-function, generating functions, Dini criterion, Jordan criterion},
pdftitle = {On the convergence of the Gaver-Stehfest algorithm},
pdfpagemode = UseNone
}\usepackage{graphicx,xspace,colortbl}
\usepackage{amsmath,amsthm,amsfonts}
\usepackage{color}
\usepackage{fancybox}
\usepackage{epsfig}
\usepackage{subfig}
\usepackage{pdfsync}
    \def\qed{\hfill$\sqcap\kern-8.0pt\hbox{$\sqcup$}$\\}
    \def\beq{\begin{eqnarray}}
    \def\eeq{\end{eqnarray}}
    \def\beqq{\begin{eqnarray*}}
    \def\eeqq{\end{eqnarray*}}

    \def\re{\textnormal {Re}}
    \def\im{\textnormal {Im}}
    
    \def\e{{\mathbb E}}
    \def\r{{\mathbb R}}
    \def\c{{\mathbb C}}
    	
    \def\d{{\textnormal d}}
    \def\i{{\textnormal i}}

\newtheorem{theorem}{Theorem}
\newtheorem{lemma}{Lemma}
\newtheorem{proposition}{Proposition}
\newtheorem{corollary}{Corollary}
\theoremstyle{definition}

\newtheorem{remark}{Remark}

\title{On the convergence of the Gaver-Stehfest algorithm}
\author{
A. Kuznetsov
\thanks{{Research supported by the
Natural Sciences and Engineering Research Council of Canada.}}  \\ \\
Dept. of Mathematics and Statistics\\  York University \\
4700 Keele Street 
\\Toronto, ON \\ M3J 1P3,  Canada 
 }

\date{Current version: \today}

\begin{document}

\maketitle

\begin{abstract}
\bigskip
The Gaver-Stehfest algorithm for numerical inversion of Laplace transform was developed in
the late 1960s. Due to its simplicity and good performance it is becoming increasingly more popular in such diverse areas as Geophysics, Operations Research and Economics, Financial and Actuarial Mathematics, Computational Physics and Chemistry. Despite the large number of applications and numerical studies, this method has never been rigorously investigated. In particular, it is not known whether the Gaver-Stehfest approximations converge and what is the rate of convergence. In this paper we answer the first of these two questions: We prove that the Gaver-Stehfest approximations converge for functions of bounded variation and functions satisfying an analogue of Dini criterion. 
\end{abstract}

{\vskip 0.5cm}
{\vskip 0.5cm}

\newpage


\section{Introduction and main results}


In this paper we are concerned with the classical problem of numerical inversion of Laplace transform. More specifically, 
assume that $f:(0,\infty)\mapsto \r$ is a locally integrable function, such that its Laplace transform 
\beqq
F(z):=\int\limits_0^{\infty} e^{-zx} f(x) \d x 
\eeqq
is finite for all $z>0$. The problem consists in recovering the original function $f(x)$ given that we know  $F(z)$.
This problem has numerous applications, and it has attracted a lot of attention from researchers over the last fifty years
(see \cite{Cohen} for an up-to-date exposition of this area). 

More specifically, we are interested in the Gaver-Stehfest algorithm, which aims to approximate $f(x)$ by a sequence of functions
\beq\label{def_Gaver_Stehfest_apprx}
f_n(x):=\ln(2) x^{-1}
 \sum\limits_{k=1}^{2n} a_k(n) F\left(k  \ln(2) x^{-1} \right), \;\;\; n\ge 1, \; x>0,
\eeq
where the coefficients are defined as follows:
\beqq
a_k(n):=
  \frac{(-1)^{n+k}}{n!} \sum\limits_{j=[(k+1)/2]}^{\min(k,n)} j^{n+1}
{n \choose j}{2j \choose j}{j \choose k-j}, \;\;\; n\ge 1, \; 1\le k \le 2n.  
\eeqq

In order to demonstrate the intuition behind these rather complicated formulas, let us explain where they come from. 
In 1966 Gaver 
\cite{Gaver1966} has introduced the following sequence of approximations 
\beq\label{def_f_k}
\tilde f_k(x):=
\ln(2) x^{-1} \frac{(2k)!}{k! (k-1)!}
\sum\limits_{i=0}^k \binom{k}{i} (-1)^i F((k+i)\ln(2) x^{-1}), \;\;\; k\ge 1, \; x>0.
\eeq
Applying the binomial theorem, the above expression can be rewritten in an equivalent integral form
\beq\label{def_f_k2}
\tilde f_k(x)=
 \int\limits_{0}^{\infty} p_k( u) f\left(\frac{x u}{\ln(2)}\right) \d u, 
\eeq
where 
\beqq
p_k(u):=\frac{(2k)!}{k! (k-1)!} (1-e^{-u})^k e^{-ku}, \;\;\; k\ge 1, \; u\ge 0. 
\eeqq
The function $p_k(u)$ is positive and its integral over $[0,\infty)$ is equal to one, 
therefore we can think of it as the density function of a positive random variable $U_k$. 
One can check by direct calculations that  
$\e[U_k]=\ln(2) + O(k^{-1})$ and ${\textnormal{Var}}(U_k)=O(k^{-1})$ 
as $k\to +\infty$. This shows that the random variables $\{U_k\}_{k\ge 1}$ converge 
in distribution to $\ln(2)$, therefore for any continuous function $f(x)$ and any $x>0$ 
we have
\beqq
\tilde f_k(x)=\e\left[f\left(\frac{xU_k}{\ln(2)}\right)\right] \to f(x), \;\;\;  k\to \infty.
\eeqq 
Gaver \cite{Gaver1966} also proves that 
\beq\label{Gaver_asymptotic_expansion}
\tilde f_k(x)=f(x)+\frac{\beta_1(x)}{k}+\frac{\beta_2(x)}{k^2}+...+
\frac{\beta_{m-1}(x)}{k^{m-1}}+O\left(k^{-m+\frac{1}{2}}\right), 
\;\;\; k\to +\infty.
\eeq
under the assumptions $f \in C^{2m}(\r^+)$ and $f^{(j)}(x) \in L_{\infty}(\r^+)$ for all $j=0,1,\dots,2m$. 
Formula \eqref{Gaver_asymptotic_expansion} shows that $\tilde f_k(x)$ converge to $f(x)$ rather slowly, 
but it also suggests using a convergence acceleration technique. 
This is exactly what was done by Stehfest \cite{Stehfest1970,Stehfest2} in 1970. 
He defined a new approximation
\beq\label{f_n_tilde_f_n}
f_n(x)=\sum\limits_{k=1}^{n} c_k(n) \tilde f_{k}(x),
\eeq
where the coefficients $c_k(n)$ are chosen so that the asymptotic terms $b_j(x)k^{-j}$, $j=1,2,\dots,n-1$ in 
\eqref{Gaver_asymptotic_expansion} are eliminated:
\beq\label{ckn_linear_conditions}
\sum\limits_{k=1}^{n} c_k(n) k^{-j}=
\begin{cases}
1, \;\;\; j=0, \\
0, \;\;\; j=1,2,\dots,n-1.
\end{cases}
\eeq
It turns out that the above conditions are satisfied by the sequence
\beqq
c_k(n):=(-1)^{n+k}\frac{k^n }{k!(n-k)!}, \;\;\; n\ge 1, 1\le k \le n.  
\eeqq
Combining the above formula with \eqref{def_f_k} and \eqref{f_n_tilde_f_n} gives us the Gaver-Stehfest approximations  
in the form  \eqref{def_Gaver_Stehfest_apprx}.

%

The Gaver-Stehfest algorithm has a number of desirable properties: (i) the approximations $f_n(x)$ are linear in values of $F(z)$; 
(ii) it requires the values of $F(z)$ for real $z$ only does not need complex numbers at all; 
(iii) the coefficients $a_k(n)$ can be easily computed;
(iv) the Gaver-Stehfest approximations are exact for constant functions, that is, 
if $f(x)\equiv c$ then $f_n(x) \equiv c$ for all $n\ge 1$.
This algorithm was studied in  \cite{Davies1979,Jacquot1983,Masol2010,Whitt06}, where it was demonstrated numerically that $f_n(x)$ converge very quickly to $f(x)$ for many examples of initial functions $f(x)$
(provided that $f(x)$ is non-oscillating). Another universally accepted fact is that this algorithm  requires high-precision arithmetic for its implementation (which is rather obvious, since the coefficients $a_k(n)$ are growing very rapidly and have alternating signs). 
 
 Over the last forty years the Gaver-Stehfest algorithm has been applied to solving various problems in Geophysics \cite{Knight1982}, Probability 
\cite{Abate_Whitt,Kou2003}, Actuarial Mathematics \cite{Badescu2005} and Mathematical Finance \cite{Schoutens2011},
Chemistry \cite{Montella2008} and Economics \cite{Kawakatsu2005}. This is just a small sample, 
and a quick search on the Internet produces hundreds of papers with references to the Gaver-Stehfest algorithm.

Despite all this popularity, there has not yet been a single rigorous investigation of this algorithm. In particular, it is not known what are the sufficient conditions on $f$ that will ensure convergence of $f_n(x)$, and what would be the rate of convergence. Stehfest
\cite{Stehfest1970} writes that ``theoretically $f_n(x)$ become the more accurate the greater $n$", and 
Proposition 8.2 in Abate and Whitt \cite{Abate_Whitt} states that for any $k>0$ we have
$f_n(x)-f(x)=o(n^{-k})$ as $n\to +\infty$, however both of these statements are not very precise and lack rigorous proof. The authors seem to assume that the derivation of the Gaver-Stehfest approximations via formulas \eqref{def_f_k}, \eqref{Gaver_asymptotic_expansion}, \eqref{f_n_tilde_f_n} and \eqref{ckn_linear_conditions} automatically gives the proof of convergence. 
This is not the case. First of all, validity of the asymptotic expansion \eqref{Gaver_asymptotic_expansion} for all $m\ge 1$ would require a very restrictive assumption
on the function $f$ (it has to be an infinitely differentiable function plus some additional assumptions). 
 The second issue is that even if the asymptotic expansion \eqref{Gaver_asymptotic_expansion} is valid for all $m\ge 1$, 
it is still not clear that $f_n(x) \to f(x)$, since the coefficients $c_k(n)$ 
grow very rapidly and have alternating signs.
Therefore, it is not at all obvious that $f_n(x)$ will converge to $f(x)$, given \eqref{f_n_tilde_f_n}, \eqref{ckn_linear_conditions} and the fact that $\tilde f_n(x)\to f(x)$.

In this paper we present the first rigorous investigation of the Gaver-Stehfest algorithm. We derive two sufficient conditions on the function $f$, which ensure convergence of $f_n(x)$. Our main results are presented in the next theorem.

\begin{theorem}\label{thm_main}
Assume that $f:(0,\infty)\mapsto \r$ is a locally integrable function, such that the Laplace transform $F(z)=\int_0^{\infty} e^{-zx} f(x) \d x$ exists for all $z>0$, and 
that $f_n(x)$ are defined by \eqref{def_Gaver_Stehfest_apprx}.
\begin{itemize}
\item[(i)] The convergence of $f_n(x)$ depends only on the values of the function $f$ in the neighborhood of $x$. 
\item[(ii)] Assume that for some $c\in \r$ and some $\epsilon \in (0,1/4)$ 
\beq\label{Dini_condition}
\int\limits_0^{\epsilon} \big| f(-x \log_2(\tfrac{1}{2}+v))+f(-x \log_2(\tfrac{1}{2}-v))-2c\big| v^{-1} \d v < \infty.
\eeq
 Then $f_n(x) \to c$ as $n\to +\infty$.
\item[(iii)] Assume that the function $f$ has bounded variation in the neighborhood of $x$. Then \\
$f_n(x) \to (f(x+0)+f(x-0))/2$ as $n\to +\infty$. 
\end{itemize}
\end{theorem}

Our sufficient conditions are very similar to the corresponding results from the theory of convergence of Fourier series. Item (i) has its counterpart in \cite[Theorem 4.1.1]{Kawata}, item (ii) should be compared with Dini  criterion \cite[Theorem 4.1.3(i)]{Kawata}
\beq\label{Dini_test}
\int\limits_0^{\epsilon} \big| f(x+v)+f(x-v)-2c\big| v^{-1} \d v < \infty,
\eeq
and item (iii) is exactly the same as Jordan criterion \cite[Theorem 4.1.3(ii)]{Kawata}.  Theorem \eqref{thm_main}(ii) also provides the following useful corollary:
\begin{corollary}
Under the assumptions of Theorem \ref{thm_main}, if 
\beqq
f(x+v)-f(x)=O(|v|^\alpha)
\eeqq
for some $\alpha>0$ and all $v$ in some neighborhood of $x$, then 
$f_n(x) \to f(x)$ as $n\to +\infty$.
\end{corollary}

We do not address the second important problem related to the rate of convergence of $f_n(x)$. We hope 
that the methods developed in this paper can be useful for solving this problem, and we leave it for future work.


\section{Proof of Theorem \ref{thm_main}}


Let us review some properties of the Lambert W-function (see \cite{Corless96}), which will be used in the proof of Theorem \ref{thm_main}. 
 We will be interested in the principal branch of the Lambert W-function, which will be denoted by $W(z)$): 
 It is an analytic function in the neighborhood of $z=0$
 and it satisfies $W(z)\exp(W(z))=z$.   

Let us investigate the range and domain of $W(z)$ in more detail. 
The function $w=x+\i y \mapsto we^w$ takes real values only if $y=0$ or $x=-y \cot(y)$ 
(see Section 4 in \cite{Corless96}). In particular, the open set 
\beqq
{\mathcal A}:=\{ w=x+\i y, \;\;\; x>- y \cot(y), \;\;\; -\pi < y < \pi  \}
\eeqq
is mapped by $z=we^w$ onto the cut complex plane $\c\setminus (-\infty, -1/e]$, see Figure \ref{fig1}. 
The function $w\mapsto we^w$ is one-to-one on the set ${\mathcal A}$, and $W(z)$ is defined as the inverse of this function. Therefore, $W(z)$ is analytic in $\c\setminus (-\infty, -1/e]$ and maps this set onto
${\mathcal A}$. 
It is known that $W(z)$ has the following explicit Taylor series at $z=0$ (see formula (3.1) in \cite{Corless96})
\beq\label{W_series}
W(z)=\sum\limits_{n\ge 1} (-n)^{n-1} \frac{z^n}{n!}, \;\;\; \vert z \vert<1/e,
\eeq
and has a branching singularity at $z=-1/e$
\beq\label{expansion_W_near_1_over_e}
W(z)=\sum_{n\ge 0} \mu_n p^n= -1+p-\frac{p^2}{3}+\frac{11}{72} p^3-\frac{43}{540} p^4+ \frac{769}{17280}p^5 
+..., 
\eeq
where $p=p(z):=\sqrt{2(1+ez)}$ and the series converges for $|p|<\sqrt{2}$ (see formula (4.22) in \cite{Corless96}).  

We will extend function $W(z)$ to the interval $\{ z\in \r, \; z<-1/e\}$ so that $W(z)$ is continuous 
in the upper half-plane $\im(z)\ge 0$. Thus $W(z)$ maps the interval $\{ z\in \r, \; z<-1/e\}$ 
onto the curve $x=-y\cot(y)$, $y \in (0,\pi)$ (which is the upper half of the black curve on Figure
\ref{fig1}(a)).  The function $W(z)$ is smooth on $\{ z\in \r, \; z<-1/e\}$, and it is clear that 
$|W(z)|$ and $\im(W(z))$ are decreasing function of $z$. 
Assuming that $z$ is real and $z<-1/e$, the equation $we^w=z$ has infinitely many solutions, which correspond to different branches of Lambert W-function. However, only two of these solutions
(given by $w=W(z)$ and $w={\overline{W(z)}}$) satisfy $|\im(w)|<\pi$. These solutions necessarily 
lie on the black curve on  Figure \ref{fig1}(a), and they will play a very important role in our investigation. 
\label{wew_discussion}

\begin{figure}
\centering
\subfloat[][$w=x+\i y$]{\label{fig1a}\includegraphics[height =5cm]{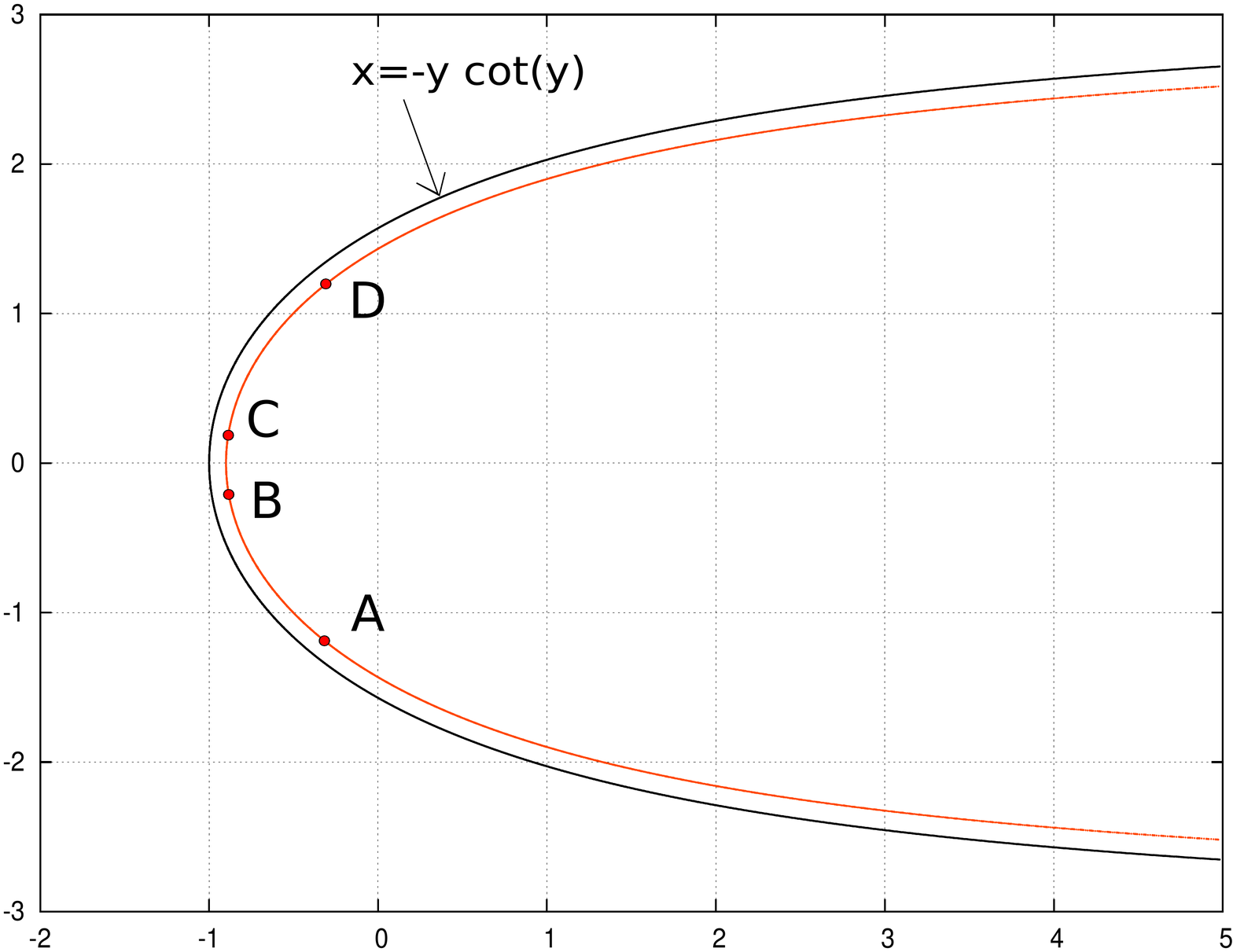}}
\subfloat[][$z=we^w$]{\label{fig1b}\includegraphics[height =5cm]{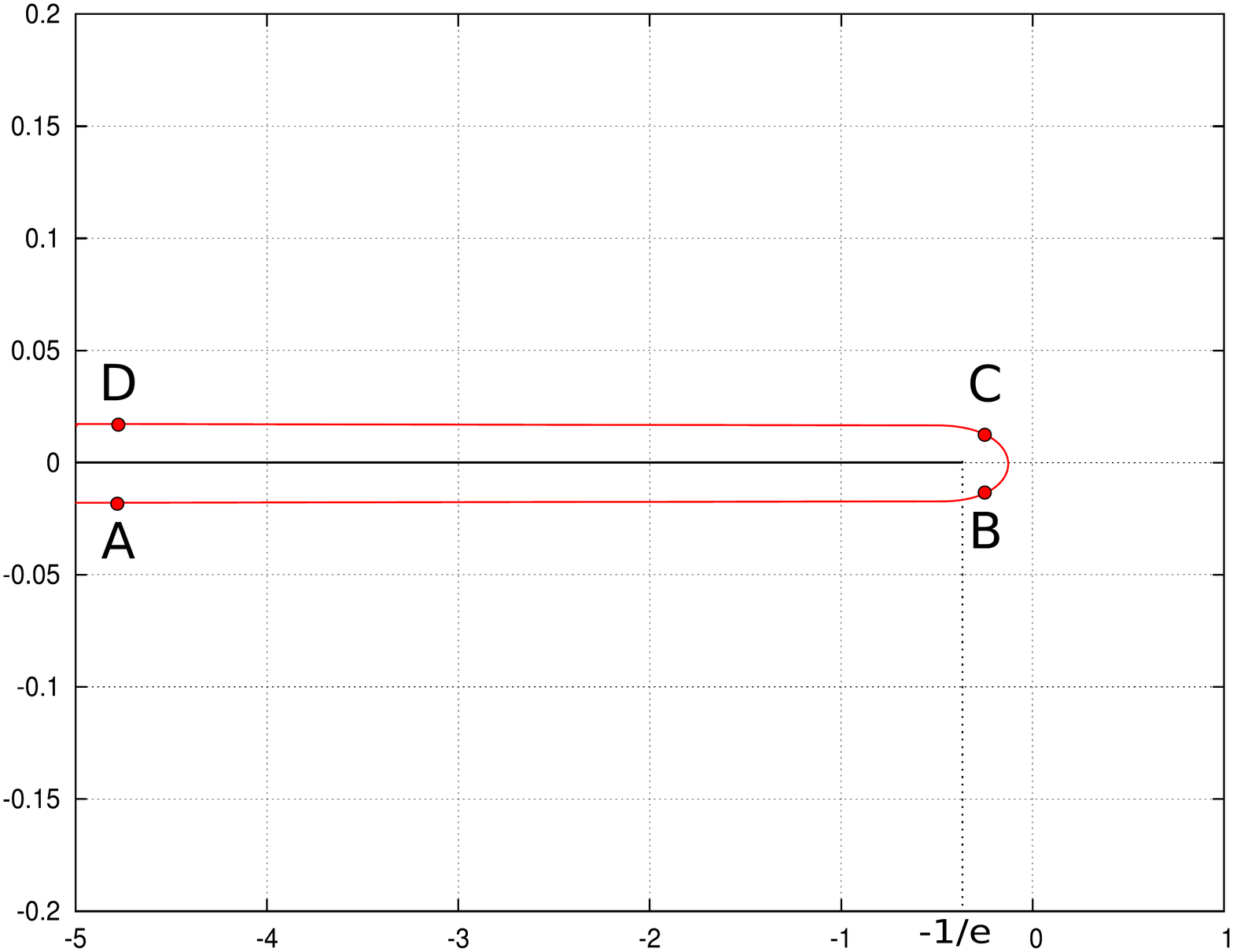}} \\
\caption{Domain (b) and range (a) of the principal branch of the Lambert W-function.}
\label{fig1}
\end{figure}

Our main tool used in the proof Theorem \ref{thm_main} is the following sequence of polynomials
\beq\label{def_qnv}
q_n(v):=\sum\limits_{k=1}^n \frac{ k^{n+1}(\tfrac{1}{2})_k}{(n-k)!(k!)^2} (-1)^{n+k} v^k, \;\;\; n\ge 1. 
\eeq
Recall that $(a)_k:=a(a+1)\dots(a+k-1)$ denotes Pochhammer symbol.   
Combining equations \eqref{def_f_k2} and \eqref{f_n_tilde_f_n} we obtain an integral representation
\beq\label{f_n_q_n_formula_1}
f_n(x)&=& \int\limits_0^{\infty}
q_n\left(4e^{-u} (1-e^{-u})\right) f\left(\frac{x u}{\ln(2)}\right) \d u,
\eeq
which explains why these polynomials are important for Gaver-Stehfest approximations. 
Our plan for proving Theorem \ref{thm_main} is to establish uniform 
asymptotics for $q_n(v)$, $0\le v \le 1$, as $n\to \infty$ and then use \eqref{f_n_q_n_formula_1} 
to study convergence of $f_n(x)$.

Let us also define 
\beq\label{def_Gz}
G(z):=\sum\limits_{n\ge 1} \frac{(\frac{1}{2})_n}{(n!)^2} (-1)^n n^{n+1} z^n. 
\eeq
Using Stirling's asymptotic formula for the Gamma function 
one can check that the above series converges for 
$|z|<1/e$.

\begin{proposition}\label{prop_generating_function_qn}
For $0\le v \le 1$ and $|t|<1/(2e)$ 
\beq\label{eqn_generating_function_qn}
G\left(vt e^t \right)=\sum\limits_{n\ge 1} q_n(v) (-1)^n t^n.
\eeq
\end{proposition}
\begin{proof}
Using \eqref{def_qnv}, the trivial estimates $(\tfrac{1}{2})_k<(1)_k=k!$ and 
$k^{n+1}<n^{n+1}$ and the binomial theorem, we obtain an upper bound 
\beqq
|q_n(v)| \le n^{n+1} \sum\limits_{k=1}^n \frac{1}{(n-k)!(k!)} <\frac{n^{n+1}2^n}{n!},
 \;\;\; {\textnormal{ for all }} \;  0\le v \le 1, 
\eeqq
which implies that the series in the left-hand side of \eqref{eqn_generating_function_qn} converges absolutely for $|t|<1/(2e)$. We combine \eqref{def_qnv} and \eqref{eqn_generating_function_qn}, interchange the order of summation and obtain
\beqq
\sum\limits_{n\ge 1} q_n(v) (-1)^n t^n&=&\sum\limits_{n\ge 1}  \sum\limits_{k=1}^n 
\frac{ k^{n+1}(\tfrac{1}{2})_k}{(n-k)!(k!)^2} (-1)^{k} v^k t^n = 
\sum\limits_{k\ge 1}\frac{(\tfrac{1}{2})_k}{(k!)^2}  (-1)^{k} k^{k+1} (vt)^k 
\sum\limits_{n\ge k} \frac{k^{n-k}t^{n-k}}{(n-k)!}\\
&=&\sum\limits_{k\ge 1}\frac{(\tfrac{1}{2})_k}{(k!)^2}  (-1)^{k} k^{k+1} (vte^t)^k =G(vte^t). 
\eeqq
\end{proof}

It is well known that the asymptotic behavior of the coefficients of the power series is closely related with the asymptotic behavior of the generating function at its dominant singularities. 
Our next goal is to obtain analytic continuation of the function $G(z)$ and to study its asymptotic expansion at the dominant singularity. 
\begin{proposition}\label{prop_analytic_continuation}
There exists a function $g(z)$, which is analytic in $\c\setminus (-\infty,-1/e]$
and continuous in the half-plane $\im(z) \ge 0$, such that
\beq\label{Gz_asymptotics}
G(z)=\frac{1}{ \sqrt{2}\pi} \left[ (1+ez)^{-1}
+ \frac{5}{24}\ln(1+ez)\right]+g(z).
\eeq
\end{proposition}
\begin{proof}
Let us denote $H(z):=-\left(z\frac{\d}{\d z} \right)^2 W(z)$.
From \eqref{W_series} we find that 
\beqq
H(z)=\sum\limits_{n\ge 1} \frac{n^{n+1}}{n!}(-1)^n z^n, \;\;\; |z|<1/e. 
\eeqq
Using the following identity (see formula 3.621.3 in \cite{Jeffrey2007}) 
\beqq
\frac{2}{\pi}\int\limits_{0}^{\frac{\pi}{2}} (\sin(t))^{2n} \d t=\frac{(\frac{1}{2})_n}{n!} 
\eeqq
and formula \eqref{def_Gz} we obtain the integral representation
\beq\label{f_F_integral}
G(z)=\frac{2}{\pi} \int\limits_0^{\frac{\pi}{2}} H(z\sin(t)^2) \d t, \;\;\; |z|<1/e. 
\eeq
The function $W(z)$ is analytic in the cut plane $\c \setminus (-\infty,-1/e]$, therefore $H(z)$ is analytic in the same domain. Applying Leibniz rule to \eqref{f_F_integral} we see that $G(z)$ is also analytic in $\c \setminus (-\infty,-1/e]$.

From \eqref{expansion_W_near_1_over_e} and the above definition of $H(z)$ we obtain 
\beq\label{Hz_series_expansion}
H(z)=\sum\limits_{n\ge -3} c_n p(z)^n=p(z)^{-3}-
\frac{11}{24} p(z)^{-1}-
\frac{4}{135}-\frac{p(z)}{1152}
+...,
\eeq
where  $p(z):=\sqrt{2(1+ez)}$ and the above series converges for $|p(z)|<\sqrt{2}$.  We subtract the dominant 
asymptotic terms from $H(z)$ and define the two functions
\beq\label{def_hz_gz}
 h(z):=H(z)-p(z)^{-3}+
\frac{11}{24} p(z)^{-1}=\sum\limits_{n\ge 0} c_n p(z)^n, \;\;\;
g_1(z):=\frac{2}{\pi} \int\limits_0^{\frac{\pi}{2}} h(z\sin(t)^2) \d t.
\eeq
It is clear that $h(z)$ is analytic in $\c \setminus (-\infty,-1/e]$ and continuous in the upper half-plane 
$\im(z)\ge 0$, therefore $g_1(z)$ satisfies the same properties.

Next, we use formula 3.681.1 in \cite{Jeffrey2007} and formula (14) on page 110 in \cite{Erdelyi1953} and conclude that for all $z\in \c \setminus (-\infty,-1/e]$
\beq\label{2F1_identities_a}
\frac{2}{\pi}\int\limits_{0}^{\frac{\pi}2} (1+ez\sin(t)^2)^{-\frac{3}{2}} \d t&=&
{}_2F_1(\tfrac{3}{2},\tfrac{1}{2};1;-ez)=\frac{2}{\pi(1+ez)}-\frac{1}{2\pi} \ln(1+ez)+g_2(z), \\
\label{2F1_identities_b}
\frac{2}{\pi}\int\limits_{0}^{\frac{\pi}2} (1+ez\sin(t)^2)^{-\frac{1}{2}} \d t&=&
{}_2F_1(\tfrac{1}{2},\tfrac{1}{2};1;-ez)=-\frac{1}{\pi} \ln(1+ez)+g_3(z),
\eeq 
where ${}_2F_1(a,b;c;z)$ denotes the hypergeometric function and both functions $g_2(z)$ and $g_3(z)$ 
 are analytic in $\c \setminus (-\infty,-1/e]$ and continuous in the half-plane 
$\im(z)\ge 0$.
Combining \eqref{f_F_integral}, \eqref{def_hz_gz}, \eqref{2F1_identities_a} and \eqref{2F1_identities_b} we see that
\beqq
G(z)&=&\frac{2}{\pi} \int\limits_0^{\frac{\pi}{2}} H(z\sin(t)^2) \d t=
\frac{2}{\pi}\int\limits_{0}^{\frac{\pi}2} p(z\sin(t)^2)^{-3} \d t-
\frac{11}{24}\times \frac{2}{\pi}\int\limits_{0}^{\frac{\pi}2} p(z\sin(t)^2)^{-1} \d t+
\frac{2}{\pi} \int\limits_0^{\frac{\pi}{2}} h(z\sin(t)^2) \d t\\
&=&
\frac{1}{\sqrt{2}\pi(1+ez)}+\frac{5}{24\sqrt{2} \pi} \ln(1+ez)+ \left[ g_1(z)+g_2(z)+g_3(z)\right],
\eeqq
which is equivalent to \eqref{Gz_asymptotics}. 
\end{proof}

\begin{lemma}\label{lemma_estimate_q_n_over_v}
For any $\epsilon \in (0,1)$ there exist constants $b=b(\epsilon)>1$ and $C=C(\epsilon)>0$ such that \\ $|q_n(v)| < C b^{-n} v$
for all $v \in [0,1-\epsilon]$. 
\end{lemma}
\begin{proof}
The function $te^t$ maps the unit circle $|t|<1$ into some open domain which is a subset of $\c \setminus (-\infty,-1/e]$ (see Figure \ref{fig1}). 
Using this fact combined with the results of Proposition \ref{prop_analytic_continuation}, we see 
that the function $t\mapsto G(te^t)$ is analytic for $|t|<1$, therefore 
there exists $R>1$ such that $t\mapsto G((1-\epsilon) te^t)$ is analytic for $|t|<R$. It is clear that for 
all $v \in [0,1-\epsilon]$ the function $t\mapsto G'(vte^t)te^t$ is also analytic for $|t|<R$. 

We differentiate both sides of \eqref{eqn_generating_function_qn} with respect to $v$ and find 
\beqq
q_n'(v)=\frac{(-1)^n}{n!} \left[ \frac{\d^n}{\d t^n}  \left(G'(vte^t) te^t \right)  \bigg|_{t=0} \; \right]. 
\eeqq
We set $b=(1+R)/2$ and use Cauchy estimates (see \cite[Theorem 2.14]{Conway})
\beqq
 \Bigg| \frac{\d^n}{\d t^n}  \left(G'(vte^t) te^t \right)  \bigg|_{t=0} \Bigg| \le n! b^{-n} M_v
\eeqq
where $M_v:=\max\{ |G'(vte^t) te^t| \; : \; |t|=b \}$. Applying the maximum modulus principle shows that
for all $v\in [0,1-\epsilon]$
\beqq
M_v &\le& \max\{ |G'(vte^t)| \; : \; |t|=b , 0\le v\le 1-\epsilon \} \times \max\{ |te^t| \; : \; |t|=b \}\\
&=&\max\{ |G'((1-\epsilon)te^t)| \; : \; |t|=b \} \times be^b=C. 
\eeqq
So far we have established that there exist $b>1$ and $C>0$ such that
$|q_n'(v)|<C b^{-n}$ for all $v \in [0,1-\epsilon]$. Since $q_n(0)=0$, we conclude that 
$|q_n(v)| \le  \int_0^v |q_n'(v)| \d v  < C b^{-n} v$
for all $v\in [0,1-\epsilon]$. 
\end{proof}

From now on, we will use the notation $w=w(v):=W(-1/(ev))$ where $v\in (0,1]$ and $W$ is principal branch of the Lambert W-function. In other words, $z=w$ is the unique solution of the equation $1+vze^{1+z}=0$ in the strip $0\le \im(z)<\pi$
(see the discussion on page \pageref{wew_discussion}, preceding
formula \eqref{def_qnv}). Note that $w(1)=-1$ and both $\im(w(v))$ and $|w(v)|$ are decreasing functions of $v \in (0,1]$ (see Figure \ref{fig1}). 

\begin{proposition}\label{asymptotics_qn}
As $n\to +\infty$ we have uniformly for all $v\in [1/2,1)$
\beq\label{qn_asymptotics}
 q_n(v)=(-1)^n \frac{\sqrt{2}}{\pi} \re \left[ \frac{w^{-n}}{(1+w)} \right]+O(w^{-n}).
\eeq
\end{proposition}
\begin{proof}
As was noted in the proof of Lemma \ref{lemma_estimate_q_n_over_v},
for any $v\in [0,1]$ the function $z\mapsto G(vze^z)$ is analytic in the circle $|z|<1$. 
Therefore,  using Cauchy integral formula and \eqref{eqn_generating_function_qn} we find that for all $v \in [1/2,1)$ 
\beq\label{q_n_integral_form}
q_n(v)=\frac{(-1)^n}{2\pi \i} \int\limits_{U_{1/2}} G(vze^z)z^{-n-1} \d z
\eeq
where the contour of integration $U_r$ denotes the circle of 
radius $r$ and center at $z=0$, winding counterclockwise around the origin. 
 
Let $g(z)$ be the function defined by \eqref{Gz_asymptotics}.  
From  \eqref{Gz_asymptotics} and \eqref{q_n_integral_form}  we find 
\beq\label{qn123}
q_n(v)=(-1)^n \left[ \frac{1}{\sqrt{2} \pi} q_{n,1}(v)+
\frac{5}{24\sqrt{2} \pi} q_{n,2}(v)+q_{n,3}(v) \right]
\eeq
 where we have defined
\beqq
q_{n,1}(v)&:=&\frac{1}{2  \pi \i} \int\limits_{U_{1/2}} \frac{z^{-n-1}}{1+vze^{1+z}} \d z,\\
q_{n,2}(v)&:=&\frac{1}{2\pi \i} \int\limits_{U_{1/2}}  \ln(1+vze^{1+z}) z^{-n-1} \d z, \\
q_{n,3}(v)&:=&\frac{1}{2\pi \i} \int\limits_{U_{1/2}} g(vze^z)z^{-n-1} \d z.
\eeqq

Our first goal is to prove that for all $v \in [1/2,1)$ we have 
\beq\label{formula_qn1}
q_{n,1}(v)=2 \re \left[ \frac{w^{-n}}{(1+w)} \right]+O(w^{-n}), \;\;\; n\to +\infty,
\eeq
and that the implied constant in $O(w^{-n})$ does not depend on $v$. 
One can check numerically that $|w(1/2)|=|W(-2/e)|\approx 1.2508<2$. In particular, for all $v \in [1/2,1)$ 
we have $1<|w(v)|<2$ and the equation $1+vze^{1+z}=0$ has exactly two solutions 
in the circle $U_{2}$, given by $z=w$ and $z=\bar w$, and no other solutions in the circle
$U_{\pi}$. This shows that for $v\in [1/2,1)$ the 
meromorphic function  $z\mapsto z^{-n-1}/(1+vze^{1+z})$ has two simple poles at $z=w$ and $z=\bar w$ and 
no other singularities in the circle $U_{\pi}$, while if $v=1$ we have a double pole at $z=w=\bar w=-1$.
We assume that $v\in [1/2,1)$, compute the residues of $z^{-n-1}/(1+vze^{1+z})$ at $z=w$ and $z=\bar w$ and applying Cauchy's residue theorem we obtain 
\beqq
\frac{1}{2\pi \i} \int\limits_{U_{1/2}} \frac{z^{-n-1}}{1+vze^{1+z}} \d z
=2 \re \left[ \frac{w^{-n}}{(1+w)} \right]+
\frac{1}{2\pi \i} \int\limits_{U_{3}} \frac{z^{-n-1}}{1+vze^{1+z}} \d z
\eeqq
The integral in the right-hand side of the above formula is estimated as follows
\beqq
\bigg | \frac{1}{2\pi \i} \int\limits_{U_{\pi}} \frac{z^{-n-1}}{1+vze^{1+z}} \d z \bigg|
\le \frac{3^{-n}}{C_1} =O(w^{-n}),
\eeqq
where $C_1:=\min\{ |1+vze^{1+z}| \; : \; 1/2\le v<1, |z|=3 \}$. Note that $C_1$ is strictly positive, since
we have shown above that for $v \in [1/2,1]$ the function $z\mapsto 1+vze^{1+z}$ has no roots in the domain $2\le|z|\le \pi$.

Our second goal is to prove that  for all $v \in [1/2,1)$ we have 
\beq\label{formula_qn2}
q_{n,2}(v)=O(w^{-n}), \;\;\; n\to +\infty. 
\eeq
We isolate the dominant singularities at $z=w$ and $z=\bar w$ and obtain
\beqq
q_n^{(2)}(v)&=&\frac{1}{2\pi \i} \int\limits_{U_{1/2}} z^{-n-1} \ln(1-z/w) \d z+
\frac{1}{2\pi \i} \int\limits_{U_{1/2}} z^{-n-1} \ln(1-z/\bar w) \d z\\
&+&
\frac{1}{2\pi \i} \int\limits_{U_{1/2}} z^{-n-1} \ln\left[\frac{1+vze^{1+z}}{(1-z/w)(1-z/\bar w)}\right] \d z.
\eeqq
The first and the second integrals can be evaluated explicitly to $(-n^{-1}w^{-n})$ and $(-n^{-1}\bar w^{-n})$ respectively, thus they contribute $O(w^{-n})$. 
The integrand in the third integral is analytic in the circle $|z|\le \pi$, 
and using the same technique as was used above for estimating $q_{n,1}(v)$, we find that 
for all $v\in [1/2,1)$ the third integral is bounded from above 
by  $O(3^{-n})=O(w^{-n})$.

Finally, for any $v\in[1/2,1]$ the function $z\mapsto  g(vze^z)$ is analytic in the circle $U_{|w|}$ and continuous on the boundary of this circle (note that $g(z)$ is continuous at $z=-1/e$ (see Proposition \ref{prop_analytic_continuation}). Therefore
\beqq
q_{n,3}(v)=\frac{1}{2\pi \i} \int\limits_{U_{1/2}} g(vze^z)z^{-n-1} \d z
=\frac{1}{2\pi \i} \int\limits_{U_{|w|}} g(vze^z)z^{-n-1} \d z
\eeqq
and we obtain the estimate
\beqq
|q_{n,3}(v)|=\bigg|\frac{1}{2\pi \i} \int\limits_{U_{|w|}} g(vze^z)z^{-n-1} \d z \bigg|\le 
|w|^{-n} C_2 
\eeqq
where $C_2:=\max \{ |g(vze^z)| \; : \; 1/2\le v \le 1, z=|w|e^{2\pi \i t}, \; 0\le t \le \pi \}$. Since 
$w(v)$ is continuous for $v\in [1/2,1]$ and $g(z)$ is continuous in the upper half-plane $\im(z)\ge 0$, we see that $C_2$ is finite and we obtain
\beq\label{formula_qn3}
q_{n,3}(v)=O(w^{-n}), \;\;\; n\to +\infty. 
\eeq 
for all $v \in [1/2,1)$.

Combining \eqref{qn123}, \eqref{formula_qn1}, \eqref{formula_qn2} and \eqref{formula_qn3} gives us \eqref{qn_asymptotics}.
\end{proof}

\begin{remark} 
With some extra work one can improve the results of Propositions \ref{prop_analytic_continuation} and \ref{asymptotics_qn} in the following way:
\begin{itemize}
\item[(i)]  The function $G(z)$ admits an asymptotic expansion at $z=-1/e$ of the form
\beqq
G(z)&\approx &\frac{1}{ \sqrt{2}\pi} \left[ (1+ez)^{-1} +  \sum\limits_{n\ge 0}  (1+ez)^n (a_n \ln(1+ez)+b_n) \right], 
\eeqq
where $a_0=5/24$, $a_1=25/1152$ and $\{a_n\}_{n\ge 2}$ are computable rational numbers. 
\item[(ii)]  As $n\to +\infty$ we have uniformly in $v\in [1/2,1)$
\beqq
 q_n(v)=(-1)^n \frac{\sqrt{2}}{\pi} \re \left[ \frac{w^{-n}}{(1+w)}-\frac{5}{24 n} w^{-n}
+\frac{25}{1152n^2}(1+w)w^{-n} 
 \right]+O(n^{-3}w^{-n})
\eeqq
where $w=W(-1/(ev))$. The above formula remains valid as $v\to 1^-$, in which case we obtain
\beqq
q_n(1)= \frac{\sqrt{2}}{\pi} \left[n+\frac{1}{3}-\frac{5}{24 n}\right]+O(n^{-3}).
\eeqq 
\end{itemize}
\end{remark}

\noindent
For $v\in [0,1/2)$ we define  
\beqq
\xi(v)=w(1-4v^2)=W(-1/(e(1-4v^2))), \;\;\; \alpha(v)=\im \left[\xi(v)\right]. 
\eeqq
These two function will play an important role in the proof of Theorem \ref{thm_main}. 
We summarize some of their properties in the next lemma.
\begin{lemma}\label{lemma_props_xi_alpha}
\mbox{}
\item[(ii)] The function $|\xi(v)|$ is smooth and strictly increasing for $v \in [0,1/2)$. It is analytic in the neighborhood of $v=0$ and satisfies
\beq\label{xi_asymptotics}
\xi(v)=-1+2\sqrt{2} \i  v + \frac{8}{3}v^2+\frac{14 \sqrt{2}}{9} \i v^3+O(v^4), \;\;\; v\to 0^+. 
\eeq 
\item[(ii)] The function $\alpha(v)$ is smooth and strictly increasing for $v \in[0,1/2)$. 
It is analytic in the neighborhood of $v=0$ and satisfies 
\beq\label{alpha_asymptotics}
\alpha(v)=2\sqrt{2}  v +\frac{14 \sqrt{2}}{9}v^3+O(v^5), \;\;\; v\to 0^+. 
\eeq 
\item[(iii)] For any function $h \in L_1(0,1/2)$ we have
\beqq
\lim\limits_{n\to +\infty} \int\limits_0^{\frac{1}{2}} h(v) \xi(v)^{-n} \d v =0. 
\eeqq
\end{lemma}
\begin{proof}
Items (i) and (ii) follow easily from properties of Lambert W-function. 
In particular the series expansion \eqref{xi_asymptotics} follows from \eqref{expansion_W_near_1_over_e}, while 
\eqref{alpha_asymptotics} is a simple corollary of \eqref{xi_asymptotics}. Item (iii) is obvious, given that
$|\xi(0)|=1$ and $|\xi(v)|$ is a strictly increasing function.  
\end{proof}

\begin{proposition}\label{prop_equivalence} 
Under the assumptions of Theorem \ref{thm_main}, $f_n(x)\to c$ as $n\to \infty$ if and only if
for any $\epsilon \in (0,1/4)$ 
\beq\label{convergence_condition}
\lim\limits_{n\to \infty} 
\int\limits_{0}^{\epsilon} |\xi(v)|^{-n} \frac{\sin(n\alpha(v))}{\alpha(v)} 
 \left[ f\left(-x\log_2\left(\tfrac{1}{2}+v\right)\right)
+ f\left(-x\log_2\left(\tfrac{1}{2}-v\right)\right)-2c \right]
 \d v =0.
\eeq
\end{proposition} 
\begin{proof}
Assume that $\epsilon \in (0,1/4)$ and define 
 $u_0:=-\ln(\tfrac{1}{2}+\epsilon)$ and $u_1:=-\ln(\tfrac{1}{2}-\epsilon)$.
Note that $0<u_0<\ln(2)<u_1<\infty$. The function $u\mapsto 4e^{-u}(1-e^{-u})$ is strictly increasing 
for $0<u<\ln(2)$ and strictly decreasing for $\ln(2)<u$, and its value at $u=\ln(2)$ is one. 
We conclude that there exists $\delta=\delta(\epsilon)>0$, such that for
 all  $u\in   [0,u_0]\cup[u_1,\infty)$ we have $0<4e^{-u}(1-e^{-u})<1-\delta$. Using Lemma 
 \ref{lemma_estimate_q_n_over_v} we find that there exist $b>1$ and $C>0$ such that
 \beqq
|q_n\left(4e^{-u} (1-e^{-u})\right)|<4 C b^{-n} e^{-u} (1-e^{-u})<4 C b^{-n} e^{-u}
 \eeqq
 for all $u\in   [0,u_0]\cup[u_1,\infty)$. This fact and our assumptions on the function $f$ guarantee that
\beqq
\lim\limits_{n\to +\infty} \int\limits_{[0,u_0]\cup[u_1,\infty) }
q_n\left(4e^{-u} (1-e^{-u})\right) f\left(\frac{x u}{\ln(2)}\right) \d u=0. 
\eeqq
Thus we have established the following result:
\beqq
f_n(x)= \int\limits_{u_0}^{u_1} 
q_n\left(4e^{-u} (1-e^{-u})\right) f\left(\frac{x u}{\ln(2)}\right) \d u+o(1), \;\;\; n\to +\infty.
\eeqq
As we have pointed out in the introduction, Gaver-Stehfest approximations are exact for constant functions: 
if $f(x)\equiv C$ then $f_n(x)\equiv C$ for all $n\ge 1$. 
This result and the above formula imply 
\beq\label{f_n_truncated_v0_v1_2}
f_n(x)-c=
\int\limits_{u_0}^{u_1} 
q_n\left(4e^{-u} (1-e^{-u})\right) \left[ f\left(\frac{x u}{\ln(2)}\right) -c \right] \d u
+o(1), \;\;\; n\to +\infty. 
\eeq

Our next goal is to simplify the integral in \eqref{f_n_truncated_v0_v1_2}. 
We separate this integral into two parts: 
the integral over $[u_0,\ln(2)]$ and the integral over $[\ln(2),u_1]$. For the first \{resp. second\} integral 
we change the variable of integration
$u=-\ln(\tfrac{1}{2}+v)$  \{resp. $u=-\ln(\tfrac{1}{2}-v)$\} and after combining the two parts we obtain
\beq\label{f_n_truncated_v0_v1_5}
f_n(x)-c=
\int\limits_{0}^{\epsilon} 
q_n(1-4v^2) 
\left[ \frac{f(-x\log_2(\tfrac{1}{2}+v))-c}{\tfrac{1}{2}+v}+
\frac{f(-x\log_2(\tfrac{1}{2}-v))-c}{\tfrac{1}{2}-v} \right] \d v+o(1), \;\;\; n\to +\infty
\eeq

For our next step, we will need the following result:
\beq\label{asymptotics_qn_14v2}
q_n(1-4v^2)=\frac{\sqrt{2}}{\pi}|\xi(v)|^{-n}\frac{\sin(n\alpha(v))}{\alpha(v)}+O(\xi(v)^{-n}),
\eeq
uniformly for all $v\in (0,1/4]$. Note that the Lambert W-function satisfies $W(z)=z\exp(-W(z))$, thus for $z<-1/e$ we have
${\textnormal{arg}}(W(z))=\pi-\im [W(z)]$, which implies ${\textnormal{arg}}(\xi(v))=\pi-\alpha(v)$ This gives us
\beqq
\xi(v)^{-n}=|\xi(v)|^{-n} e^{- \i n (\pi-\alpha(v))}= (-1)^n |\xi(v)|^{-n} \left(\cos(n \alpha(v))+\i \times \sin(n \alpha(v))\right). 
\eeqq
Combining the above formula with the identity 
\beqq
\frac{1}{1+\xi}=\frac{1+\bar {\xi}}{|1+\xi|^2}=\frac{1+\re(\xi)}{|1+\xi|^2}-\i \frac{\im(\xi)}{|1+\xi|^2}
\eeqq
we obtain
\beqq
\re\left[ \frac{\xi(v)^{-n}}{1+\xi(v)} \right]=|\xi(v)|^{-n} \left[\cos(n\alpha(v)) 
\frac{1+\re(\xi(v))}{|1+\xi(v)|^2} +
\sin(n\alpha(v)) 
\frac{\alpha(v)}{|1+\xi(v)|^2}\right].
\eeqq
Formulas \eqref{xi_asymptotics} and \eqref{alpha_asymptotics} show that uniformly for all $v\in (0,1/4]$ 
\beqq
\frac{1+\re(\xi(v))}{|1+\xi(v)|^2}=O(1), \;\;\;
\frac{\alpha(v)}{|1+\xi(v)|^2}=\frac{1}{\alpha(v)}+O(1)
\eeqq
thus 
\beqq
\re\left[ \frac{\xi(v)^{-n}}{1+\xi(v)} \right]=(-1)^{n}|\xi(v)|^{-n}\frac{\sin(n\alpha(v))}{\alpha(v)}+O(\xi(v)^{-n})
\eeqq
which together with \eqref{qn_asymptotics} implies \eqref{asymptotics_qn_14v2}.

Next, we define
\beq \label{def_gxvc}
g(x,v,c)&:=& f\left(-x\log_2\left(\tfrac{1}{2}+v\right)\right)
+ f\left(-x\log_2\left(\tfrac{1}{2}-v\right)\right)-2c, \\ \nonumber
\tilde g(x,v,c)&:=& -\frac{f\left(-x\log_2\left(\frac{1}{2}+v\right)\right)-c}{\frac{1}{2}+v}
+  \frac{f\left(-x\log_2\left(\frac{1}{2}-v\right)\right)-c}{\frac{1}{2}-v}.
\eeq
It is easy to check that 
\beqq
\left[ \frac{f(-x\log_2(\tfrac{1}{2}+v))-c}{\tfrac{1}{2}+v}+
\frac{f(-x\log_2(\tfrac{1}{2}-v))-c}{\tfrac{1}{2}-v} \right]\equiv 2g(x,v,c)+2v \tilde g(x,v,c),
\eeqq
therefore
\beq\label{f_n_truncated_v0_v1_3}
f_n(x)-c=
2\int\limits_{0}^{\epsilon} q_n(1-4v^2)  g(x,v,c) \d v+
2\int\limits_{0}^{\epsilon} q_n(1-4v^2) v \tilde g(x,v,c) \d v+o(1), \;\;\; n\to +\infty.
\eeq
We use \eqref{asymptotics_qn_14v2} and find that
\beqq
\int\limits_{0}^{\epsilon} q_n(1-4v^2) v \tilde g(x,v,c) \d v
=\frac{\sqrt{2}}{\pi}
\int\limits_{0}^{\epsilon} \left[ |\xi(v)|^{-n} \sin(n\alpha(v)) \frac{v}{\alpha(v)}+v O(\xi(v)^{-n}) \right] \tilde g(x,v,c) \d v\to 0,
\eeqq
as $n\to +\infty$, due to Lemma \ref{lemma_props_xi_alpha}(iii) (note that the function
$v\mapsto \tilde g(x,v,c)$ is integrable and the function $v/\alpha(v)$ is bounded).  
This fact combined with formula \eqref{f_n_truncated_v0_v1_3} give us \eqref{convergence_condition}. 
\end{proof}

\vspace{0.15cm}
\noindent
{\bf Proof of Theorem \ref{thm_main}, (i) and (ii):}
Theorem \ref{thm_main}, (i) follows directly from Proposition \ref{prop_equivalence}.
Let us prove Theorem \ref{thm_main}, (ii).  Formulas \eqref{convergence_condition} and \eqref{def_gxvc}
give us
\beqq
f_n(x)-c&=&
2\int\limits_{0}^{\epsilon} |\xi(v)|^{-n} \frac{\sin(n\alpha(v))}{\alpha(v)} g(x,v,c) \d v+o(1)\\
&=&
2\int\limits_{0}^{\epsilon} |\xi(v)|^{-n} \left[\sin(n\alpha(v))\times \frac{v}{\alpha(v)} \times \frac{g(x,v,c)}{v} \right] \d v+o(1).
\eeqq
The functions $\sin(n\alpha(v))$ and $v/\alpha(v)$ are bounded  
and the function $v\mapsto g(x,v,c)/v$ is integrable over $[0,\epsilon)$, thus applying Lemma \ref{lemma_props_xi_alpha}(iii) we see that the integral in the right-hand side of the above formula converges 
to zero as $n\to+\infty$, therefore $f_n(x)\to c$ as $n\to +\infty$. 
\qed

\vspace{0.15cm}
\noindent
{\bf Proof of Theorem \ref{thm_main},(iii):}
Assume that for some $\delta>0$ the function $f$ has bounded variation on the interval $[x-\delta,x+\delta]$. We will need the following simple property of functions of bounded variation:

\vspace{0.1cm}
\noindent
{\it If $f(y)$ has bounded variation for $y\in [a,b]$ and $g(x)$ is monotone for $x\in [c,d]$ and satisfies 
$g([c,d]) \subseteq [a,b]$, then $f(g(x))$ has bounded variation for $x\in [c,d]$.}
 \vspace{0.1cm}
 
 \noindent
Using the above property we see that there exists $\epsilon_1 \in (0,1/4)$ such that both functions 
$v\mapsto f(-x \log_2(\tfrac{1}{2}- v))$ and $v\mapsto f(-x \log_2(\tfrac{1}{2}+v))$ 
have bounded variation in the interval $v\in [0,\epsilon_1]$. 

Formula \eqref{alpha_asymptotics} shows that there exists $\epsilon_2 \in (0,1/4)$, such that
 $\alpha''(v)>0$ for all $v \in (0,\epsilon_2)$.  We set $\epsilon_3=\min(\epsilon_1, \epsilon_2)$
 and take $\epsilon<\epsilon_3$ to be a small positive number (to be specified later). 
 We rewrite formula \eqref{convergence_condition} where $\epsilon$ is chosen in the above way
and  $c=\frac{1}{2} \left( f(x+0)+f(x-0)\right)$:  
\beq\label{eqn_In_pm} 
f_n(x)-c&=&\int\limits_{0}^{\epsilon} |\xi(v)|^{-n} \frac{\sin(n\alpha(v))}{\alpha(v)} 
\left[ f(-x \log_2(\tfrac{1}{2}-v))-f(x+0) \right]\d v
\\&+& \nonumber
\int\limits_{0}^{\epsilon} |\xi(v)|^{-n} \frac{\sin(n\alpha(v))}{\alpha(v)} 
\left[ f(-x \log_2(\tfrac{1}{2}+v))-f(x-0) \right]\d v+o(1), \;\;\; n\to +\infty.
\eeq
Since the two functions $v\mapsto f(-x \log_2(\tfrac{1}{2}-v))$ and $v\mapsto f(-x \log_2(\tfrac{1}{2}+v))$  have bounded variation 
for $v\in [0,\epsilon_3]$, there exist four increasing functions $h_i(v)$, $1\le i \le 4$, satisfying 
$h_i(0+)=0$, such that $f(-x \log_2(\tfrac{1}{2}-v))-f(x+0)=h_1(v)-h_2(v)$ 
and $f(-x \log_2(\tfrac{1}{2}+v))-f(x-0)=h_3(v)-h_4(v)$ 
for all $v\in [0,\epsilon_3]$. We rewrite \eqref{eqn_In_pm}  in the form
\beq\label{f_n_J_n}
f_n(x)-c=J_{n,1}-J_{n,2}+J_{n,3}-J_{n,4}+o(1), \;\;\; n\to +\infty, 
\eeq
where 
\beq\label{def_Jni}
J_{n,i}:=\int\limits_{0}^{\epsilon} |\xi(v)|^{-n} \frac{\sin(n\alpha(v))}{\alpha(v)} 
h_i(v)\d v. 
\eeq
Since $h_1$ is positive and increasing, we apply the mean value theorem and conclude that there exists
 $\theta_1 \in (0,\epsilon)$ such that
\beq\label{Jni_f1}
J_{n,1}=\int\limits_0^{\epsilon} 
|\xi(v)|^{-n}\frac{\sin(n\alpha(v))}{\alpha(v)} h_1(v) \d v=
h_1(\epsilon)\int\limits_{\theta_1}^{\epsilon} 
|\xi(v)|^{-n}\frac{\sin(n\alpha(v))}{\alpha(v)}  \d v.
\eeq
Recall that $\alpha''(0)>0$ for $v\in [0,\epsilon_3]$
and $\alpha'(0)=2\sqrt{2}$ (see \eqref{alpha_asymptotics}), which shows that $\alpha'(v)$ is positive and increasing for $v\in [0,\epsilon_3]$, which in turn implies that 
the function $|\xi(v)|^{-n}/\alpha'(v)$ is positive and decreasing for $v\in [0,\epsilon_3]$. Applying the mean-value theorem again, we find that
there exists some $\theta_2 \in (\theta_1,\epsilon)$ such that
\beq\label{Jni_f2}
\nonumber
\int\limits_{\theta_1}^{\epsilon} 
|\xi(v)|^{-n}\frac{\sin(n\alpha(v))}{\alpha(v)}  \d v
&=&
\int\limits_{\theta_1}^{\epsilon} 
\left[\frac{|\xi(v)|^{-n}}{\alpha'(v)} \right] \times \left[\frac{\sin(n\alpha(v))}{\alpha(v)} \alpha'(v)\right] \d v
=\frac{|\xi(\theta_1)|^{-n}}{\alpha'(\theta_1)} 
\int\limits_{\theta_1}^{\theta_2} 
\frac{\sin(n\alpha(v))}{\alpha(v)} \alpha'(v) \d v\\
&=&
\frac{|\xi(\theta_1)|^{-n}}{\alpha'(\theta_1)}
\int\limits_{n\alpha(\theta_1)}^{n\alpha(\theta_2)} 
\frac{\sin(u)}{u} \d u=
\frac{|\xi(\theta_1)|^{-n}}{\alpha'(\theta_1)} \left[{\textnormal{Si}}(n\alpha(\theta_1))-
{\textnormal{Si}}(n\alpha(\theta_2))\right]
\eeq
where ${\textnormal{Si}}(y):=\int_0^y \sin(u) u^{-1} \d u$ is the sine-integral function. 
Since $|\xi(\theta_1)|>1$ and $\alpha'(\theta_1)>2\sqrt{2}$ 
and $|{\textnormal{Si}}(x)|<\pi$ for all $x\ge 0$, we conclude from 
\eqref{def_Jni}, \eqref{Jni_f1}  and \eqref{Jni_f2} that 
\beqq
|J_{n,1}|<\frac{\pi}{\sqrt{2}} h_1(\epsilon).
\eeqq
In the same way we can obtain corresponding estimates for $J_{n,i}$, $i=2,3,4$, therefore 
\eqref{f_n_J_n} gives us
\beqq
\limsup\limits_{n\to +\infty} |f_n(x)-c|< \frac{\pi}{\sqrt{2}} (h_1(\epsilon)+h_2(\epsilon)+h_3(\epsilon)+h_4(\epsilon)).
\eeqq
Sine $h_i(0+)=0$, the right-hand side can be made arbitrarily small provided that $\epsilon$ is small enough. This shows that $f_n(x) \to c$ as $n\to +\infty$. 
\qed



\end{document}